\documentclass[12pt]{amsart}

\usepackage{slashed}
\usepackage{amsrefs, mathrsfs,amsmath, amscd, amsthm,amssymb, amsfonts, verbatim,subfigure, enumerate,stmaryrd, pxfonts}
\usepackage[mathscr]{euscript}
\usepackage[all,cmtip]{xy}
\usepackage{graphicx}
\usepackage[top=1in, bottom=1.25in, left=1.3in, right=1.3in]{geometry}
\usepackage{mathpazo}
\usepackage[colorlinks=true,linkcolor=blue,citecolor=red]{hyperref}

\usepackage[all,cmtip]{xy}
\newcommand{\xym}{\xymatrix@1@=14pt@M=2pt}

\usepackage{tikz}
\usepackage{tikz-cd}
\usetikzlibrary{arrows,shapes}
\usetikzlibrary{trees}
\usetikzlibrary{matrix,arrows}
\usetikzlibrary{positioning}
\usetikzlibrary{calc,through}
\usetikzlibrary{decorations.pathreplacing}
\usepackage{pgffor}
\usepackage{pgfkeys}
\usepackage{genyoungtabtikz}

\usetikzlibrary{decorations.pathmorphing}
\usetikzlibrary{decorations.markings}
\tikzset{
    vector/.style={decorate, decoration={snake}, draw},
	provector/.style={decorate, decoration={snake,amplitude=2.5pt}, draw},
	antivector/.style={decorate, decoration={snake,amplitude=-2.5pt}, draw},
    fermion/.style={draw=black, postaction={decorate},
        decoration={markings,mark=at position .55 with {\arrow[draw=black]{>}}}},
    fermionbar/.style={draw=black, postaction={decorate},
        decoration={markings,mark=at position .55 with {\arrow[draw=black]{<}}}},
    fermionnoarrow/.style={draw=black},
    gluon/.style={decorate, draw=black,
        decoration={coil,amplitude=4pt, segment length=5pt}},
    scalar/.style={dashed,draw=black, postaction={decorate},
        decoration={markings,mark=at position .55 with {\arrow[draw=black]{>}}}},
    scalarbar/.style={dashed,draw=black, postaction={decorate},
        dwecoration={markings,mark=at position .55 with {\arrow[draw=black]{<}}}},
    scalarnoarrow/.style={dashed,draw=black},
    electron/.style={draw=black, postaction={decorate},
        decoration={markings,mark=at position .55 with {\arrow[draw=black]{>}}}},
	bigvector/.style={decorate, decoration={snake,amplitude=4pt}, draw},
}

\newtheorem{theorem}{Theorem}[section]
\newtheorem{prop}[theorem]{Proposition}
\newtheorem{lemma}[theorem]{Lemma}

\theoremstyle{definition}
\newtheorem{dfn}[theorem]{Definition}
\newtheorem{dfn/lem}{Definition/Lemma}

\theoremstyle{remark}
\newtheorem{rmk}[theorem]{Remark}
\newtheorem{eg}[theorem]{Example}

\usepackage{parskip}
\usepackage{color}   %May be necessary if you want to color links
\usepackage{hyperref}
\hypersetup{
    colorlinks=true, %set true if you want colored links
    linktoc=all,     %set to all if you want both sections and subsections linked
    linkcolor=blue,  %choose some color if you want links to stand out
}

\newcommand{\nc}{\newcommand}
\nc{\on}{\operatorname}
\nc{\mc}{\mathcal}

\nc{\A}{\on{A}}
\nc{\ZZ}{\mathbb{Z}}
\nc{\PP}{\mathbb{P}}
\nc{\mt}{\langle t \rangle}
\nc{\mxn}{\langle x_1 ,...,  x_n \rangle }
\nc{\Pn}{\mathbb{P}_{n}}
\nc{\Zplus}{\mathbb{Z}_{\geq 0}}

\nc{\Sk}{\on{Skew}}

\nc{\funmod}{\on{\fun-mod}}
\nc{\ra}{\rightarrow}

\nc{\Mod}{\on{Mod}}

\nc{\Amod}{\on{Mod}(\A)_{\fun}}
\nc{\mtmod}{\on{Mod}(\mt)_{\fun}}
\nc{\mxnmod}{\on{Mod}(\Pn)_{\fun}}
\nc{\C}{\mathcal{C}}

\nc{\fun}{\mathbb{F}_1}

\nc{\Ksplit}{\on{K}_{0}^{sp}}
\nc{\bch}{\Phi_{k}}
\nc{\Adj}{\on{Adj}}

\nc{\wt}{\widetilde}

\allowdisplaybreaks[4]  %%%%%%%%%%% choose 1-4, where 4 is the strongest desire to breack

\DeclareMathOperator{\lcm}{lcm}

\begin{document}

 \title{ Split Grothendieck rings of rooted trees and skew shapes via monoid representations.  }
  \author{David Beers and Matt Szczesny}
  \date{}

\begin{abstract}
We study commutative ring structures on the integral span of rooted trees and $n$-dimensional skew shapes. The multiplication in these rings arises from the smash product operation on monoid representations in pointed sets. We interpret these as Grothendieck rings of indecomposable monoid representations over $\fun$ - the "field" of one element. We also study the base-change homomorphism from $\mt$-modules to $k[t]$-modules for a field $k$ containing all roots of unity, and interpret the result in terms of Jordan decompositions of adjacency matrices of certain graphs. 
\end{abstract} 

  \maketitle

\section{Introduction}

In this paper we consider commutative ring structures on the integral spans of rooted trees and $n$--dimensional skew shapes. The product in these rings arises by first interpreting the corresponding combinatorial structure as a representation of a monoid in pointed sets, and then using the smash product, which defines a symmetric monoidal structure on the category of such representations. We proceed to explain the construction in greater detail. 

To a monoid $\A$, one may associate a category $\Amod$ of "representations of $\A$ over the field of one element", whose objects are finite pointed sets with an action of $\A$. The terminology comes from the general yoga of $\fun$, where pointed sets are viewed as vector spaces over $\fun$, and monoids are viewed as non-additive analogues of algebras (see \cite{CLS, L1}).  Given $\Amod$, their categorical coproduct $M \oplus N$ is given by the wedge sum $M \vee N$ and the product by the Cartesian product $M \times N$ (equipped with diagonal $\A$-action). One may also consider a reduced version of the Cartesian product -- the smash product $M \wedge N$, with $\A$--action $a(m \wedge n) = am \wedge an$, which while not a categorical product, defines a symmetric monoidal structure on  $\Amod$. $\oplus$ and $\wedge$ are compatible in the sense that $$M \wedge (K \oplus L) \simeq M \wedge K \oplus M \wedge L.$$

In certain cases, objects of $\Amod$ have a pleasant interpretation in terms of familiar combinatorial structures. For example, when $\A = \mt$, the free monoid on one generator $t$, we may associate to $M \in \mtmod$ a graph $\Gamma_M$ which encodes the action of $t$ on $M$. The vertices of $\Gamma_M$ correspond to the non-zero elements of $M$ (where the basepoint plays the role of zero), and the directed edges join $m \in M$ to $t \cdot m$. The possible connected graphs arising this way, corresponding to indecomposable representations, are easily seen to be of two types - rooted trees and wheels:

\begin{minipage}{.5\textwidth}
\begin{center}
\begin{tikzpicture}
\draw [ultra thick,->] (0,0) -- (0.9,0.9);
\draw [fill] (0,0) circle [radius=0.1];
\draw [fill] (1,1) circle [radius=0.1];
\draw [ultra thick,->] (2,0) -- (1.1,0.9);
\draw [fill] (2,0) circle [radius=0.1];
\draw [ultra thick,->] (1,1) -- (1.9,1.9);
\draw [fill] (2,2) circle [radius=0.1];
\draw [fill] (4,0) circle [radius=0.1];
\draw [fill] (3,1) circle [radius=0.1];
\draw [ultra thick,->] (4,0) -- (3.1, 0.9);
\draw [ultra thick,->] (3,1) -- (2.1,1.9);
\draw [fill] (2,1) circle [radius=0.1];
\draw [ultra thick,->] (2,1) -- (2,1.9);
\draw [fill] (2,3) circle [radius=0.1];
\draw [ultra thick,->] (2,2) -- (2,2.9);
\draw [fill] (3,2) circle [radius=0.1];
\draw [ultra thick,->] (3,2) -- (2.1,2.9);
\draw [ultra thick,->] (2,3) -- (2,3.9);
\draw [fill] (2,4) circle [radius=0.1];
\node at (2,-1) {Rooted tree};
\end{tikzpicture}
\end{center}
\end{minipage}
\begin{minipage}{.5\textwidth}
\begin{center}
\begin{tikzpicture}
\draw [ultra thick,->] (0,0) -- (0.9,0);
\draw [ultra thick,->] (1,0) -- (1,0.9);
\draw [ultra thick,->] (1,1) -- (0.1,1);
\draw [ultra thick,->] (0,1) -- (0,0.1);
\draw [ultra thick,->] (0,-1) -- (0,-0.1);
\draw [ultra thick,->] (-1,-1) -- (-0.1,-0.1);
\draw [ultra thick,->] (2,2) -- (1.1,1.1);
\draw [ultra thick,->] (2,3) -- (2,2.1);
\draw [ultra thick,->] (3,2) -- (2.1,2);
\draw [fill] (0,0) circle [radius=0.1];
\draw [fill] (1,1) circle [radius=0.1];
\draw [fill] (1,0) circle [radius=0.1];
\draw [fill] (0,1) circle [radius=0.1];
\draw [fill] (-1,-1) circle [radius=0.1];
\draw [fill] (0,-1) circle [radius=0.1];
\draw [fill] (2,2) circle [radius=0.1];
\draw [fill] (2,3) circle [radius=0.1];
\draw [fill] (3,2) circle [radius=0.1];
\node at (1,-2) {Wheel};
\end{tikzpicture}
\end{center}
\end{minipage}

Given indecomposable $M,N \in \mtmod$ (corresponding to a tree or wheel), one can ask how $\Gamma_{M \wedge N}$ can be computed from $\Gamma_M$ and $\Gamma_N$. We give the answer in Section \ref{mt}, in the form of a simple algorithm, and show that $\Gamma_{M \wedge N}$ corresponds to the tensor product of graphs $\Gamma_M \otimes \Gamma_N$ in the sense of \cite{Weich}. 

In a similar vein, $n$--dimensional skew shapes can be interpreted as representations of $\mxn$ - the free commutative monoid on $n$ generators $x_1, \cdots, x_n$. We illustrate this for $n=2$, where the shape $S$

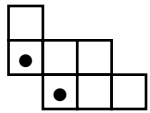
\begin{center}
\Ylinethick{1pt}
\gyoung(;,\bullet;;,:;\bullet;;)
\end{center}
determines a module over the the free commutative monoid on two generators $ \langle x_1, x_2 \rangle$, whose non-zero elements correspond to the boxes in the diagram.   $x_1$ acts by  moving one box to the right, and $x_2$ by moving one box up, until the edge of the diagram is reached, and by $0$ beyond that. Connected skew shapes yield indecomposable representations of $\mxn$, and we may once again ask how to decompose $M_S \wedge M_T$ into $\oplus_i M_{U_i}$, where $U_i$ are connected skew shapes. The answer is given in Section \ref{skew}, where we prove the following theorem:

\begin{theorem}
If $S_1$ and $S_2$ $n$-dimensional skew shapes, then
\begin{equation*}
M_{S_1} \wedge M_{S_2} = \bigoplus_{t \in \mathbb{Z}^n} M_{S_1 \cap (S_2 + t)}
\end{equation*}
\end{theorem}

In other words, the $U_i$ are those skew shapes that occur in the intersection of one shape with a translate of the other. 

Our results may be phrased in a more structured way as follows. Given a  monoid $\A$, and a monoidal sub-category $\C \subset (\Amod, \wedge)$, we may consider the split Grothendieck ring $K^{split}(\C)$. Elements of $K^{split}(\C)$ may be identified with formal integer linear combinations $\sum a_i [M_i]$ of isomorphism classes of $[M_i] \in \on{Iso}(\C)$, subject to the relations 
\[
[M \oplus N] \sim [M] + [N],
\]
with multiplication induced by the smash product. In our examples, $K^{split}(\C)$ consists of integer linear combinations of trees/wheels or skew shapes. The results of this paper amount to an explicit combinatorial description of the product in $K^{split}(\C)$. 

Structures over $\fun$ may be based-changed to those over a field (or any commutative ring) $k$. We denote this functor $\otimes_{\fun} k$. $\A \otimes_{\fun} k$ is the monoid algebra $k[\A]$, and for $M \in \Amod$, $M \otimes_{\fun} k$ the $k[\A]$-module spanned over $k$ by elements of $M$. $k[\A]$ is a $k$-bialgebra, and so its category of modules monoidal. The functor $\otimes_{\fun} k$ is monoidal, and so induces a ring homomorphism $$\Phi_k : \Ksplit(\Amod) \ra \Ksplit(\on{Mod}_{k[\A]}).$$ We study this homomorphism in Section \ref{bch} in the simple case of the monoid $\A=\mt$, in which case generators of $\Ksplit(\Mod(k[t]))$ can be identified with Jordan blocks. Understanding $\Phi_k$ in this case reduces to computing the Jordan form of the adjacency matrices of the trees/wheels above. We show the image of $\Phi_k$ is spanned by nilpotent Jordan blocks and cyclotomic diagonal matrices.

\subsection{Outline of paper}
Section \ref{MM} recalls basic facts regarding monoids and the category $\Amod$, and define the split Grothendieck ring $\Ksplit(\Amod)$. In Section \ref{mt} we consider the example of $\A = \mt$ - the free monoid on one generator, and identify the product in $\Ksplit(\mtmod)$ with the graph tensor product of trees/wheels. In Section \ref{bch} we consider the base-change homomorphism $\bch: \Ksplit(\mtmod) \ra \Ksplit(\on{Mod}_{k[t]})$ and describe its image in terms of the Jordan decomposition of the adjacency matrix of the corresponding graph. Section \ref{skew} is devoted to the example of $\A = \Pn = \mxn$ - the free commutative monoid on $n$ generators, and a certain subcategory of $\mxnmod$ corresponding to $n$-dimensional skew shapes. We give an explicit description of the product in $\Ksplit(\mxnmod)$ in terms of intersections of skew shapes. 

\noindent {\bf Acknowledgements:} This paper emerged from an undergraduate research project at Boston University completed by the first author with the second as faculty mentor. We gratefully acknowledge the generous support of the BU UROP program during the research and writing phase of this project. The second author is supported by a Simons Foundation Collaboration Grant.

\section{Monoids and their modules} \label{MM}

A \emph{monoid} $\A$ will be an associative semigroup with identity $1_A$ and zero $0_A$ (i.e. the absorbing element). We require
\[
 1_A \cdot a = a \cdot 1_A = a \hspace{1cm} 0_A \cdot a = a \cdot 0_A = 0_A \hspace{1cm} \forall a \in A
\]
Monoid homomorphisms are required to respect the multiplication as well as the special elements $1_A, 0_A$.

% An \emph{ideal} of $A$ is a nonempty subset  $\a \subset A$ such that $A \cdot \a \subset \a$. An proper ideal $\p \subset A$ is \emph{prime} if $xy \in \p$ implies either $x \in \p$ or $y \in \p$. 

\begin{eg}
Let $\fun = \{ 0, 1\}$ with $$ 0 \cdot 1 = 1 \cdot 0 = 0 \cdot 0 = 0 \textrm{ and } 1 \cdot 1 = 1 .$$ We call $\fun$ \emph{the field with one element}. 
\end{eg}

\begin{eg}
Let $$\Pn := \mxn = \{ x^{r_1}_1 x^{r_2}_2 \cdots x^{r_n}_n \vert r=(r_1, r_2, \cdots, r_n) \in \Zplus^n \} \cup \{ 0 \},$$ the set of monomials in $x_1, \cdots, x_n$, with the usual multiplication. We will often write elements of $\Pn$ in multiindex notation as $x^r, r \in \Zplus^n$, in which case the multiplication is written as $$ x^r \cdot x^s = x^{r+s} .$$ We identify $x^0$ with $1$. $\Pn$ has a natural $\ZZ^n_{\geq 0}$-grading obtained by setting $deg(x_i) = e_i$ - the $i$th standard basis vector in $\ZZ^n$. 
\end{eg}

$\fun$ and $\Pn$ are both commutative monoids. 

\subsection{The category $\Amod$}

\begin{dfn}
Let $\A$ be a monoid. An \emph{$\A$-module} is a pointed set $(M,0_M)$ (with $0_M \in M$ denoting the basepoint), equipped with an action of $\A$. More explicitly, an $\A$-module structure on $(M, 0_M)$ is given by a map
\begin{align*}
\A \times M  & \rightarrow M \\
(a, m) & \rightarrow a \cdot m
\end{align*}
satisfying 
\[
(a \cdot b)\cdot m = a \cdot (b \cdot m), \; \; \; 1 \cdot m = m, \; \; \; 0 \cdot m = 0_M, \; \; \; a \cdot 0_M = 0_M, \; \; \forall a,b, \in \A, \; m \in M
\]
\end{dfn}
A \emph{morphism} of $\A$-modules is given by a pointed map $f: M \rightarrow N$ compatible with the action of $\A$, i.e. $f(a \cdot m) = a \cdot f(m)$. 
The $\A$-module $M$ is said to be \emph{finite} if $M$ is a finite set, in which case we define its \emph{dimension} to be $dim(M) = \vert M \vert -1$ (we do not count the basepoint, since it is the analogue of $0$). We say that $N \subset M$ is an \emph{$\A$--submodule} if it is a (necessarily pointed) subset of $M$ preserved by the action of $\A$. $\A$ always posses the module $\mathbb{O} := \{0\}$, which will be referred to as the \emph{zero module}, as well as the \emph{trivial module} $\mathbb{1} :=  \fun$, on which all non-zero elements of $\A$ act by the identity (this arises via the augmentation homomorphism $\A \ra \fun$ sending all non-zero elements to $1$). 
\medskip

\noindent {\bf Note:} This structure is called an \emph{$\A$-act} in \cite{KKM} and an \emph{$\A$-set} in \cite{CLS}. 

\medskip

We denote by $\Amod$ the category of finite $\A$-modules. It is the $\fun$ analogue of the category of a finite-dimensional representations of an algebra. Note that for $M \in \Amod$, $\on{End}_{\Amod}(M) := \on{Hom}_{\Amod}(M,M)$ is a monoid (in general non-commutative). An $\fun$-module is simply a pointed set, and will be referred to as a vector space over $\fun$. Thus, an $\A$-module structure on $M \in \funmod$ amounts to a monoid homomorphism $A \rightarrow \on{End}_{\funmod}(M)$. 

\medskip

Given a morphism $f: M \rightarrow N$ in $\Amod$, we define the \emph{image} of $f$ to be $$Im(f) := \{ n \in N \vert \exists m \in M, f(m) = n \}.$$ For $M \in \Amod$ and an $\A$--submodule $N \subset M$, the \emph{quotient} of $M$ by $N$, denoted $M/N$ is the $\A$-module $$ M/N :=  M \backslash N \cup \{0 \}, $$ i.e. the pointed set obtained by identifying all elements of $N$ with the base-point, equipped with the induced $\A$--action. 

\bigskip

We recall some properties of $\Amod$,  following \cite{KKM, CLS, SzSgp}, where we refer the reader for details:

\medskip

\begin{enumerate}
\item For $M,N \in \Amod$, $\vert Hom_{\Amod}(M,N) \vert < \infty$ \label{part1}
\item The trivial $\A$-module $0$ is an initial, terminal, and hence zero object of $\Amod$. 
\item Every morphism $f: M \rightarrow N$ in $C_A$ has a kernel $Ker(f):=f^{-1}(0_N)$.
\item  Every morphism $f: M \rightarrow N$ in $C_A$ has a cokernel $Coker(f):=M/Im(f)$. 
\item The co-product of a finite collection $\{ M_i \}, i \in I$ in $\Amod$ exists, and is given by the wedge product
$$
\bigvee_{i \in I} M_i = \coprod M_i / \sim
$$
where $\sim$ is the equivalence relation identifying the basepoints. We will denote the co-product of $\{M_i \}$ by $$\oplus_{i \in I} M_i$$
\item The product of a finite collection $\{ M_i \}, i \in I$ in $\Amod$ exists, and is given by the Cartesian product $\prod M_i$, equipped with the diagonal $\A$--action. It is clearly associative. It is however not compatible with the coproduct in the sense that $M \times (N \oplus L) \nsimeq M \times N \oplus M \times L$.
\item The category $\Amod$ possesses a reduced version $M \wedge N$ of the Cartesian product $M \times N$, called the smash product. $M \wedge N := M \times N / M \vee N$, where $M$ and $N$ are identified with the $\A$--submodules $\{ (m,0_N) \}$ and $\{(0_M,n)\}$ of $M \times N$ respectively. The smash product inherits the associativity from the Cartesian product, and is compatible with the co-product - i.e. $$M \wedge (N \oplus L) \simeq M \wedge N \oplus M \wedge L.$$ It defines a symmetric monoidal structure on $\Amod$, with unit $\fun$ (i.e. $M \wedge \fun \simeq M$).  
\item $\Amod$ possesses small limits and co-limits. 
%\item If $\A$ is commutative, $\Amod$ acquires a monoidal structure called the \emph{tensor product}, denoted $M \otimes_{\A} N$, and defined by
%\[
%M \otimes_{\A} N := M \times N / \sim_{\otimes}
%\]
%where $\sim_{\otimes}$ is the equivalence relation generated by $(a \cdot m, n) \sim_{\otimes} (m, a \cdot n)$ for all $a \in \A, m \in M, n \in N$. Note that $(0_M,n) = (0 \cdot 0_M, n) \sim_{\otimes} (0_M, 0 \cdot n) = (0_M,0_N)$, and likewise $(m,0_N) \sim_{\otimes} (0_M,0_N)$. This allows us to rewrite the tensor product as $M \otimes_{\A} N = M \wedge N/ \sim_{\otimes'}$, where $\sim_{\otimes'}$ denotes the equivalence relation induced on $M \wedge N$ by $\sim_{\otimes}$. We have 
%\begin{align*}
%M \otimes_{\A} N  & \simeq N \otimes_{\A} M, \\  (M \otimes_{\A} N) \otimes_{\A} L & \simeq M \otimes_{\A} (N \otimes_{\A} L), \\ \ M \otimes_{\A} (L \oplus N) & \simeq (M \otimes_{\A} L) \oplus (M \otimes_{\A} N). 
%\end{align*}
\item Given $M$ in $\Amod$ and $N \subset M$, there is an inclusion-preserving correspondence between flags $N \subset L \subset M$ in $\Amod$ and $\A$--submodules of $M/N$ given by sending $L$ to $L/N$. The inverse correspondence is given by sending $K \subset M/N$ to $\pi^{-1} (K)$, where $\pi: M \rightarrow M/N$ is the canonical projection. This correspondence has the property that if $N \subset L \subset L' \subset M$, then $(L'/N)/(L/N) \simeq L'/L$.  \label{property9}
\end{enumerate}

\medskip

\noindent These properties suggest that $\Amod$ has many of the properties of an abelian category, without being additive. It is an example of a \emph{quasi-exact} and \emph{belian} category in the sense of Deitmar \cite{D3} and a \emph{proto-abelian} category in the sense of Dyckerhoff-Kapranov \cite{DK}. Let $\on{Iso}(\Amod)$ denote the set of isomorphism classes in $\Amod$, and by $[M]$ the isomorphism class of $M \in \Amod$.

We will regard $\Amod$ as a symmetric monoidal category with respect to $\wedge$ and unit $\fun$. 

\begin{dfn}

\begin{enumerate}
\item We say that $M \in \Amod$ is \emph{indecomposable} if it cannot be written as $M = N \oplus L$ for non-zero $N, L \in \Amod$. 
\item We say $M \in \Amod$ is \emph{irreducible} or \emph{simple} if it contains no proper sub-modules (i.e those different from $0$ and $M$). 
\end{enumerate}
\end{dfn}

It is clear that every irreducible module is indecomposable. 
%Given $m \in  M$, denote by $\A \cdot m := \{ a \cdot m \vert a \in \A \}$ its orbit under the action of $\A$. It is clear that $M$ is indecomposable iff it consists of a single $\A$--orbit, and that in general $M$ can be uniquely written in the form $\oplus^{k}_{i = 1} \A \cdot m_i$ for $m_1, \cdots, m_k \in M$.
We have the following analogue of the Krull-Schmidt theorem (\cite{SzSgp}): 

\begin{prop} \label{Krull_Schmidt}
Every $M \in \Amod$ can be uniquely decomposed (up to reordering) as a direct sum of indecomposable $\A$-modules. 
\end{prop}

%\begin{proof}
%Let $\Omega_M$ be the directed colored graph with vertex set $ V(\Omega_M) := M \backslash *$, and edge set  $$ E(\Omega_M) := \{ m \rightarrow a \cdot m \vert a \in A, a \cdot m \neq * \},$$ where the edge $m \rightarrow a \cdot m$ is colored by $a$. It is clear that
%\begin{equation} \label{disj_union}
%\Omega_{M \oplus N} = \Omega_M \coprod \Omega_N
%\end{equation}
% Let $\Omega_M = \Gamma_1 \cup \Gamma_2 \cdots \cup \Gamma_k$ be the decomposition of $\Omega_M$ into connected components, and $M_{\Gamma_i}$ the subset of $M$ corresponding to $\Gamma_i$, together with the basepoint $*$. Then $M = \oplus M_{\Gamma_k}$, and it is clear  from \ref{disj_union}  that each $M_{\Gamma_i}$ is indecomposable. The uniqueness of the decomposition is immediate. 
%\end{proof}

%\bigskip

\begin{rmk} \label{subobj}

Suppose $M = \oplus^k_{i=1} M_i$ is the decomposition of an $\A$-module into indecomposables, and $N \subset M$ is a submodule. It then immediately follows that $N = \oplus (N \cap M_i)$. 

\end{rmk}

%Let $\on{Rep}(\A) := \mathbb{Z}[ \ol{M} ]/I ,\; \ol{M} \in \on{Iso}(\Amod) $, where $I$ is the sub-group generated by differences $\ol{M \oplus N} - \ol{M} - \ol{N}$. The fact that the symmetric monoidal operations $\wedge, \otimes$ (when defined) are compatible with $\oplus$ shows that they descend to $\on{Rep}(\A)$. More precisely:

%\begin{definition}
%Let $\A$ be a semigroup. 
%\begin{enumerate}
%\item Let $\on{Rep}^{\wedge}(\A)$ denote the commutative ring obtained from  $\on{Rep}(\A)$ using the product induced by the smash product on $\Amod$.
%\item If $\A$ is commutative, let $\on{Rep}^{\otimes}(\A)$ denote the commutative ring obtained from $\on{Rep}(\A)$ using the product induced by the tensor product on $\Amod$.
%\end{enumerate}
%\end{definition}

%\bigskip

\subsection{Monoid algebras}

In this section, we recall a few facts regarding monoid algebras following \cite{Steinberg}. Let $k$ be a field. The monoid algebra $k[\A]$ consists of linear combinations of non-zero elements of $\A$ with coefficients in $k$. I.e.
\[
k[\A] = \{ \sum c_a a \vert a \in \A, a \neq 0, c_a \in k \}
\]
with product induced from the product in $\A$, extended $k$--linearly. $k[\A]$ is a bialgebra, with co-product
\[
\Delta: k[\A] \ra k[\A] \otimes k[\A]
\]
determined by
\[
\Delta(a) = a \otimes a, \; a \in \A
\]
The category $\on{Mod}_{k[\A]}$ of $k[\A]$-modules is therefore symmetric monoidal under the operation of tensoring over $k$. 

There is a base-change functor: 
\begin{equation}
\otimes_{\fun} k : \Amod \rightarrow \Mod_{k[\A]}
\end{equation}
to the category of $k[\A]$--modules defined by setting
\[
M \otimes_{\fun} k := \bigoplus_{m \in M, m \neq 0_M} k \cdot m 
\]
i.e. the free $k$--module on the non-zero elements of $M$, with the $k[\A]$-action induced from the $\A$--action on $M$. It sends $f \in \on{Hom}_{\A}(M,N)$ to its unique $k$--linear extension in $\on{Hom}_{k[\A]} (M \otimes_{\fun} k, N \otimes_{\fun} k)$. 

We will find the following elementary observation useful:

\begin{prop} \label{monoidal_base_change}
The functor $\otimes_{\fun} k:  \Amod \rightarrow \Mod_{k[\A]} $ is monoidal.
\end{prop}

As a consequence, we have that for $M, N \in \Amod$, $$(M \wedge N) \otimes_{\fun} k \simeq (M\otimes_{\fun} k) \otimes_k (N \otimes_{\fun} k)$$ as $k[\A]$-modules. 

\subsection{The split Grothendieck ring}

\begin{dfn} \label{Ksplit_def}
The \emph{split Grothendieck ring} of $\Amod$, denoted $\Ksplit(\Amod)$ is the $\ZZ$--linear span of isomorphism classes in $\Amod$ modulo the relation $[M \oplus N] = [M] + [N]$. I.e.
\[
\Ksplit(\Amod) = \ZZ \left[ [M] \right]/I  \;\;\; [M] \in \on{Iso}(\Amod)
\]
where $I$ is the ideal generated by all differences $[M\oplus N] - [M] - [N]$, with product induced by $\wedge$. Since by Prop \ref{Krull_Schmidt} every module is a direct sum of indecomposable ones, we can also describe $\Ksplit{\Amod}$ as the $\ZZ$-linear span of indecomposable $\A$-modules: 
\begin{equation} \label{lin_comb}
\Ksplit(\Amod) := \{ \sum a_i [M_i] \vert a_i \in \ZZ, [M_i] \in \on{Iso}(\Amod), M_i \textrm{ is indecomposable } \}
\end{equation}
with the product of two isomorphism classes $[M], [M']$ of indecomposables given by
\[
[M] \cdot [M'] = \sum [N_i] \, \textrm{ if } M \wedge M' \simeq \oplus N_i, \; \; N_i \textrm{ indecomposable } 
\]
\end{dfn}

We note that $\Ksplit(\Amod)$ is a commutative ring with identity the isomorphism class $[\fun]$ of the trivial $\A$-module. 

More generally, if $\C$ is a subcategory of $\Amod$ closed under $\oplus$ and $\wedge$, we may consider $\Ksplit(\C)$, where the span in \ref{lin_comb} is restricted to the indecomposable modules in $\C$. 

The following is an immediate consequence of the of the functor $\otimes_{\fun} k$ being monoidal:

\begin{prop}
There is a ring homormorphism
\[
\bch: \Ksplit(\Amod) \ra \Ksplit(\Mod_{k[\A]})
\]
\end{prop}

\section{Rooted trees, wheels, and the monoid $\mt$}
In this section we study the ring $\Ksplit(\Amod)$ in the case where $\A=\mt$, the free monoid on one generator,  and the corresponding base-change homomorphism 
\[
\bch: \Ksplit(\Amod) \ra \Ksplit(\Mod_{k[t]})
\]
for a field $k$. Recall that finite-dimensional $k[t]$-modules correspond to pairs $(V, T)$ where $V$ is a finite-dimensional vector space over $k$, and $T \in \on{End}(V)$. The indecomposable $k[t]$-modules thus correspond to Jordan blocks. It follows by analogy that the study of finite $\mt$-modules amounts to studying "linear algebra over $\fun$", and the indecomposable $\mt$-modules are the corresponding Jordan blocks over $\fun$. 

Given $M \in \mtmod$, we may associate to it a graph $\Gamma_M$ which encodes the action of $t$ on $M$. The vertices of $\Gamma_M$ correspond bijectively to the non-zero elements of $M$, and the directed edges join $m \in M$ to $t \cdot m$. We will make no distinction between $m \in M$ and the corresponding vertex of $\Gamma_M$ when the context is clear. 

The possible connected graphs arising as $\Gamma_M$, corresponding to indecomposable $\mt$-modules, were classified in \cite{SzSgp} and are easily seen to be of two types:

\begin{minipage}{.5\textwidth}
\begin{center}
\begin{tikzpicture}
\draw [ultra thick,->] (0,0) -- (0.9,0.9);
\draw [fill] (0,0) circle [radius=0.1];
\draw [fill] (1,1) circle [radius=0.1];
\draw [ultra thick,->] (2,0) -- (1.1,0.9);
\draw [fill] (2,0) circle [radius=0.1];
\draw [ultra thick,->] (1,1) -- (1.9,1.9);
\draw [fill] (2,2) circle [radius=0.1];
\draw [fill] (4,0) circle [radius=0.1];
\draw [fill] (3,1) circle [radius=0.1];
\draw [ultra thick,->] (4,0) -- (3.1, 0.9);
\draw [ultra thick,->] (3,1) -- (2.1,1.9);
\draw [fill] (2,1) circle [radius=0.1];
\draw [ultra thick,->] (2,1) -- (2,1.9);
\draw [fill] (2,3) circle [radius=0.1];
\draw [ultra thick,->] (2,2) -- (2,2.9);
\draw [fill] (3,2) circle [radius=0.1];
\draw [ultra thick,->] (3,2) -- (2.1,2.9);
\draw [ultra thick,->] (2,3) -- (2,3.9);
\draw [fill] (2,4) circle [radius=0.1];
\node at (2,-1) {Rooted tree};
\end{tikzpicture}
\end{center}
\end{minipage}
\begin{minipage}{.5\textwidth}
\begin{center}
\begin{tikzpicture}
\draw [ultra thick,->] (0,0) -- (0.9,0);
\draw [ultra thick,->] (1,0) -- (1,0.9);
\draw [ultra thick,->] (1,1) -- (0.1,1);
\draw [ultra thick,->] (0,1) -- (0,0.1);
\draw [ultra thick,->] (0,-1) -- (0,-0.1);
\draw [ultra thick,->] (-1,-1) -- (-0.1,-0.1);
\draw [ultra thick,->] (2,2) -- (1.1,1.1);
\draw [ultra thick,->] (2,3) -- (2,2.1);
\draw [ultra thick,->] (3,2) -- (2.1,2);
\draw [fill] (0,0) circle [radius=0.1];
\draw [fill] (1,1) circle [radius=0.1];
\draw [fill] (1,0) circle [radius=0.1];
\draw [fill] (0,1) circle [radius=0.1];
\draw [fill] (-1,-1) circle [radius=0.1];
\draw [fill] (0,-1) circle [radius=0.1];
\draw [fill] (2,2) circle [radius=0.1];
\draw [fill] (2,3) circle [radius=0.1];
\draw [fill] (3,2) circle [radius=0.1];
\node at (1,-2) {Wheel};
\end{tikzpicture}
\end{center}
\end{minipage}

We call the first type a \emph{rooted tree} and the second a \emph{wheel}. Rooted trees correspond to indecomposable $\mt$-modules where $t$ acts nilpotently, in the sense that $t^n \cdot m = 0$ for sufficently large $n$. We call such a module \emph{nilpotent}. 

We will use the following terminology when discussing the graphs $\Gamma_M$
\begin{itemize}
\item We call a vertex with no outgoing edges a \emph{root}. It is drawn at the top. A connected $\Gamma_M$ can have at most one root. 
\item If $M$ is nilpotent, hence $\Gamma_M$ a tree, then the \emph{depth} of a vertex $m \neq 0$, denoted $depth(m)$ is the number of edges in the unique path connecting $m$ to the root. The only vertex of depth zero is the root. In general, $depth(m)+1$ is the smallest power of $t$ that annihilates $m$.
\item The \emph{height} of a rooted tree is the maximal depth of any of its vertices. The tree in the above example has height $4$.
\item A \emph{cycle of length} $n$ is a sequence of distinct elements $Z = \{ m_1, \cdots, m_n \}, m_i \in M$, such that $t \cdot m_i = m_{i+1}$ and $t \cdot m_n = m_1$. 
\item A \emph{chain of length} $n$ is a sequence of distinct elements $C = \{ m_1, m_2, \cdots, m_n \}, m_i \in M$, such that $t\cdot m_i = m_{i+1}, 1 \geq i < n$, but $t \cdot m_n \neq m_1$.
\end{itemize}

Wheels contain a single directed cycle, possibly with trees attached. A wheel is easily seen to arise from a $\mt$-module $M$ where $t^r \cdot m = t^{r+n} \cdot m$ for some $r, n \in \mathbb{N}$ for every $m \in M$. 
 
We begin with the problem of computing the product in $\Ksplit(\mtmod)$ in terms of the graphs above.

\subsection{Products in  $\Ksplit(\mtmod)$  }  \label{mt}

Given a $\mt$-module $M$, and $m \in M$, we define 
\[
\on{pred}(m) = \{ m' \in M, t \cdot m' = m \}
\]
At the level of the graph $\Gamma_M$, $\on{pred}(m), m \neq 0$ corresponds to the vertices connected to $m$ via directed edge. Recall that for $M, N \in \mtmod$ and $(m,n) \in M \wedge N$, $t \cdot (m,n) = (t\cdot m, t \cdot n)$. In particular, $t \cdot (m,n) = 0$ iff $t \cdot m = 0$ or $t \cdot n = 0$.  The following observations are immediate:

\begin{prop} \label{sp}
Let $M, N \in \mtmod$ be indecomposable. 
\begin{enumerate}
\item $M \wedge N$ is nilpotent iff at least one of $M, N$ is nilpotent.
\item If $M, N$ are nilpotent, and $(m,n) \in M \wedge N$, then $depth((m,n))= min(depth(m), depth(n))$.
\item If $M$ is nilpotent, and $N$ is not, then for $(m,n) \in M \wedge N$, $depth((m,n)) = depth(m)$. 
\item $\on{pred}(0) \subset M = \on{ker}(t)$, and corresponds to a root in the corresponding component of $\Gamma_M$. 
\item For $(m,n) \in M \wedge N$,
 \begin{equation*}
  \on{pred} (m,n) = \{ (m', n') \vert m' \in \on{pred}(m), n' \in \on{pred}(n) \}.
  \end{equation*}
  I.e. $\on{pred}(m,n) = \on{pred}(m) \times \on{pred}(n)$.
\item $\{ \on{pred}(0) \subset M \wedge N \}   = \lbrace \{ \on{pred}(0) \subset M \} \times N \rbrace \cup  \lbrace M \times  \{ \on{pred}(0) \subset N \} \rbrace.$ 
\end{enumerate}
\end{prop}

We proceed to examine the three cases where each of $\Gamma_M, \Gamma_N$ is a rooted tree/wheel. 

\begin{itemize}
\item If $\Gamma_M, \Gamma_N$ are both rooted trees, then $\Gamma_{M \wedge N}$ consists of $dim(M) + dim(N) - 1$ rooted trees whose roots correspond to pairs $(m,n) \in M \wedge N$ where at least one of $m, n$ is a root. Each component has height $\leq min(height(\Gamma_M), height(\Gamma_N))$, and at least one component where the inequality is sharp. 
\item If $\Gamma_M$ is a tree and $\Gamma_N$ is a wheel, then $\Gamma_{M \wedge N}$ consists of $dim(N)$ rooted trees whose roots correspond to pairs $(r_M, n)$ where $r_M$ is the root of $\Gamma_M$.  Each component has height $\leq height(\Gamma_M)$. 
\item If $\Gamma_M, \Gamma_N$ are both wheels containing cycles of length $l_M, l_N$, then $ker(t)=0$ in both $M$ and $N$, and so $ker(t)=0$ on $M \wedge N$. Each connected component of $\Gamma_{M \wedge N}$ is therefore a wheel, and contains a unique cycle. If $(m,n) \in M \wedge N$ is part of a cycle, then 
\begin{equation} \label{product_cycle}
t^r \cdot (m,n) = (m,n)
\end{equation}
for some $r$, which implies that $t^r \cdot m = m$ and $t^r \cdot n = n$. It follows that $m$ (resp. $n$) is itself part of a cycle in $\Gamma_M$ (resp. $\Gamma_N$). Moreover, $r$ must be a multiple of $l_M$ and $l_N$. Since the length of the cycle containing $(m,n)$ is the least $r$ such that equation \ref{product_cycle} holds, it follows that $r=\lcm(l_M, l_N)$.

To summarize, have thus shown that each connected component of $\Gamma_{M \wedge N}$ contains a (necessarily unique) cycle of length $\lcm(l_M, l_N)$, and that $(m,n)$ occurs in a cycle iff $m, n$ do as well. Since there are $l_M l_N$ such pairs, it follows that $\Gamma_{M \wedge N}$ has $\frac{l_M l_N}{\lcm(l_M, l_N)} = \gcd(l_M, l_N)$ connected components. 
 \end{itemize}

We note that each connected component of $\Gamma_{M \wedge N}$ is determined recursively by property $(5)$ above. For instance, if at least one of $\Gamma_M, \Gamma_N$ is a rooted tree, we may begin with a vertex $(r_M, n)$ or $(m, r_N)$ corresponding to a root in $\Gamma_{M \wedge N}$ and build the rest of the component using $(5)$. The same approach works if both graphs are wheels, though there is no preferred choice for the starting vertex. 

\begin{eg}
The two trees $\Gamma_N$ and $\Gamma_M$ yield the forest $\Gamma_{N \wedge M}$ pictured below, with $6$ connected components, each of which has height $\leq 1$. 

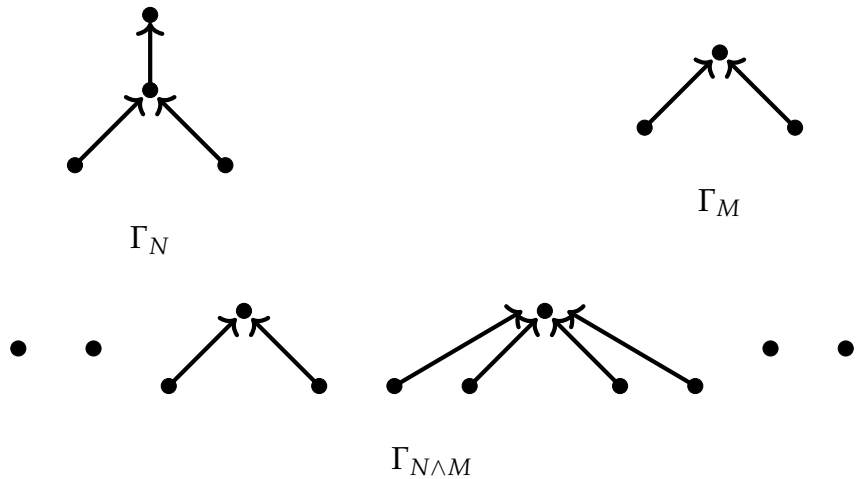
\begin{minipage}{.5\textwidth}
\begin{center}
\begin{tikzpicture}
\draw [fill] (-1,0) circle [radius=0.1];
\draw [fill] (1,0) circle [radius=0.1];
\draw [fill] (0,1) circle [radius=0.1];
\draw [fill] (0,2) circle [radius=0.1];

\draw [ultra thick,->] (-1,0) -- (-0.1,0.9);
\draw [ultra thick,->] (1,0) -- (0.1,0.9);
\draw [ultra thick,->] (0,1) -- (0,1.9);

\node at (0,-1) {$\Gamma_N$};
\end{tikzpicture}
\end{center}
\end{minipage}
\begin{minipage}{.5\textwidth}
\begin{center}
\begin{tikzpicture}
\draw [fill] (-1,0) circle [radius=0.1];
\draw [fill] (1,0) circle [radius=0.1];
\draw [fill] (0,1) circle [radius=0.1];

\draw [ultra thick,->] (-1,0) -- (-0.1,0.9);
\draw [ultra thick,->] (1,0) -- (0.1,0.9);

\node at (0,-1) {$\Gamma_M$};
\end{tikzpicture}
\end{center}
\end{minipage}

\begin{center}
\begin{tikzpicture}
\draw [fill] (-5,0.5) circle [radius=0.1];
\draw [fill] (-4,0.5) circle [radius=0.1];
\draw [fill] (-3,0) circle [radius=0.1];
\draw [fill] (-1,0) circle [radius=0.1];
\draw [fill] (0,0) circle [radius=0.1];
\draw [fill] (1,0) circle [radius=0.1];
\draw [fill] (3,0) circle [radius=0.1];
\draw [fill] (4,0) circle [radius=0.1];
\draw [fill] (5,0.5) circle [radius=0.1];
\draw [fill] (6,0.5) circle [radius=0.1];

\draw [fill] (-2,1) circle [radius=0.1];
\draw [fill] (2,1) circle [radius=0.1];

\draw [ultra thick,->] (-3,0) -- (-2.1,0.9);
\draw [ultra thick,->] (-1,0) -- (-1.9,0.9);

\draw [ultra thick,->] (0,0) -- (1.7,1);
\draw [ultra thick,->] (1,0) -- (1.9,0.9);
\draw [ultra thick,->] (3,0) -- (2.1,0.9);
\draw [ultra thick,->] (4,0) -- (2.3,1);

\node at (0.5,-1) {$\Gamma_{N \wedge M}$};
\end{tikzpicture}
\end{center}
\end{eg}

\begin{eg}
The tree $\Gamma_N$ and the wheel $\Gamma_M$ yield the forest $\Gamma_{N \wedge M}$ pictured below, with $3$ connected components, each of which has height $\leq 2$. 

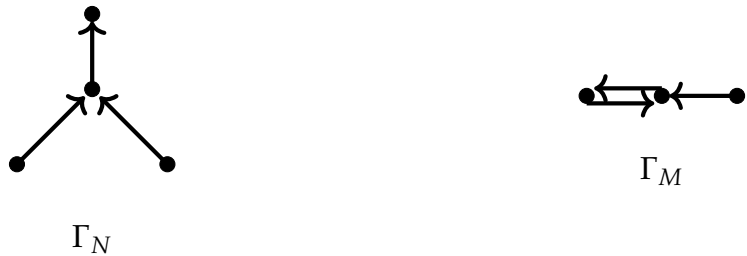
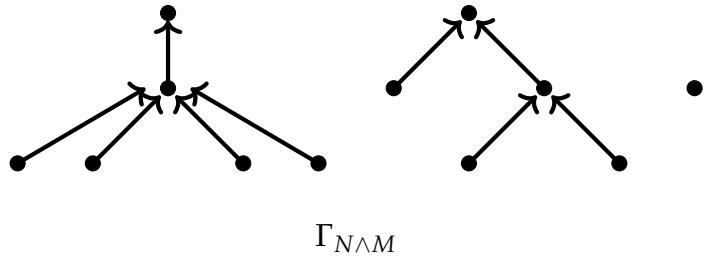
\begin{minipage}{.5\textwidth}
\begin{center}
\begin{tikzpicture}
\draw [fill] (-1,0) circle [radius=0.1];
\draw [fill] (1,0) circle [radius=0.1];
\draw [fill] (0,1) circle [radius=0.1];
\draw [fill] (0,2) circle [radius=0.1];

\draw [ultra thick,->] (-1,0) -- (-0.1,0.9);
\draw [ultra thick,->] (1,0) -- (0.1,0.9);
\draw [ultra thick,->] (0,1) -- (0,1.9);

\node at (0,-1) {$\Gamma_N$};
\end{tikzpicture}
\end{center}
\end{minipage}
\begin{minipage}{.5\textwidth}
\begin{center}
\begin{tikzpicture}
\draw [fill] (-1,0) circle [radius=0.1];
\draw [fill] (0,0) circle [radius=0.1];
\draw [fill] (1,0) circle [radius=0.1];

\draw [ultra thick,->] (-1,-0.1) -- (-.1,-0.1);
\draw [ultra thick,->] (0,0.1) -- (-0.9,0.1);
\draw [ultra thick,->] (1,0) -- (0.1,0);

\node at (0,-1) {$\Gamma_M$};
\end{tikzpicture}
\end{center}
\end{minipage}

\begin{center}
\begin{tikzpicture}
\draw [fill] (-3,0) circle [radius=0.1];
\draw [fill] (-2,0) circle [radius=0.1];
\draw [fill] (0,0) circle [radius=0.1];
\draw [fill] (1,0) circle [radius=0.1];
\draw [fill] (3,0) circle [radius=0.1];
\draw [fill] (5,0) circle [radius=0.1];

\draw [fill] (-1,1) circle [radius=0.1];
\draw [fill] (2,1) circle [radius=0.1];
\draw [fill] (4,1) circle [radius=0.1];
\draw [fill] (6,1) circle [radius=0.1];

\draw [fill] (-1,2) circle [radius=0.1];
\draw [fill] (3,2) circle [radius=0.1];

\draw [ultra thick,->] (-3,0) -- (-1.3,1);
\draw [ultra thick,->] (-2,0) -- (-1.1,0.9);
\draw [ultra thick,->] (0,0) -- (-0.9,0.9);
\draw [ultra thick,->] (1,0) -- (-0.7,1);
\draw [ultra thick,->] (3,0) -- (3.9,0.9);
\draw [ultra thick,->] (5,0) -- (4.1,0.9);

\draw [ultra thick,->] (-1,1) -- (-1,1.9);
\draw [ultra thick,->] (2,1) -- (2.9,1.9);
\draw [ultra thick,->] (4,1) -- (3.1,1.9);

\node at (1.5,-1) {$\Gamma_{N \wedge M}$};
\end{tikzpicture}
\end{center}
\end{eg}

\begin{eg}
The two wheels $\Gamma_N$ and $\Gamma_M$ yield $\Gamma_{N \wedge M}$ pictured below, with $\gcd(2,2)=2$ wheels, each with a cycle of $\lcm(2,2)=2$ vertices.

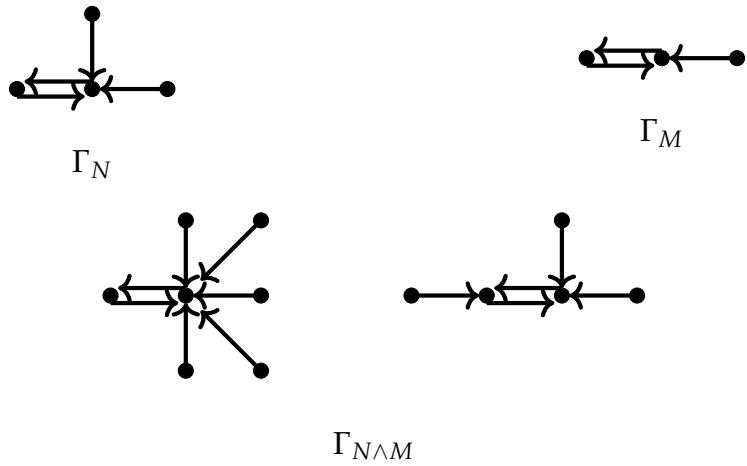
\begin{minipage}{.5\textwidth}
\begin{center}
\begin{tikzpicture}
\draw [fill] (-1,0) circle [radius=0.1];
\draw [fill] (0,0) circle [radius=0.1];
\draw [fill] (1,0) circle [radius=0.1];
\draw [fill] (0,1) circle [radius=0.1];

\draw [ultra thick,->] (-1,-0.1) -- (-.1,-0.1);
\draw [ultra thick,->] (0,0.1) -- (-0.9,0.1);
\draw [ultra thick,->] (1,0) -- (0.1,0);
\draw [ultra thick,->] (0,1) -- (0,0.1);

\node at (0,-1) {$\Gamma_N$};
\end{tikzpicture}
\end{center}
\end{minipage}
\begin{minipage}{.5\textwidth}
\begin{center}
\begin{tikzpicture}
\draw [fill] (-1,0) circle [radius=0.1];
\draw [fill] (0,0) circle [radius=0.1];
\draw [fill] (1,0) circle [radius=0.1];

\draw [ultra thick,->] (-1,-0.1) -- (-.1,-0.1);
\draw [ultra thick,->] (0,0.1) -- (-0.9,0.1);
\draw [ultra thick,->] (1,0) -- (0.1,0);

\node at (0,-1) {$\Gamma_M$};
\end{tikzpicture}
\end{center}
\end{minipage}

\begin{center}
\begin{tikzpicture}
\draw [fill] (-1,0) circle [radius=0.1];
\draw [fill] (0,0) circle [radius=0.1];
\draw [fill] (1,0) circle [radius=0.1];
\draw [fill] (0,1) circle [radius=0.1];
\draw [fill] (1,1) circle [radius=0.1];
\draw [fill] (0,-1) circle [radius=0.1];
\draw [fill] (1,-1) circle [radius=0.1];

\draw [ultra thick,->] (-1,-0.1) -- (-.1,-0.1);
\draw [ultra thick,->] (0,0.1) -- (-0.9,0.1);
\draw [ultra thick,->] (1,0) -- (0.1,0);
\draw [ultra thick,->] (0,1) -- (0,0.1);
\draw [ultra thick,->] (1,1) -- (0.2,0.2);
\draw [ultra thick,->] (0,-1) -- (0,-0.1);
\draw [ultra thick,->] (1,-1) -- (0.2,-0.2);

\draw [fill] (3,0) circle [radius=0.1];
\draw [fill] (4,0) circle [radius=0.1];
\draw [fill] (5,0) circle [radius=0.1];
\draw [fill] (6,0) circle [radius=0.1];
\draw [fill] (5,1) circle [radius=0.1];

\draw [ultra thick,->] (3,0) -- (3.9,0);
\draw [ultra thick,->] (4,-0.1) -- (4.9,-0.1);
\draw [ultra thick,->] (5,0.1) -- (4.1,0.1);
\draw [ultra thick,->] (6,0) -- (5.1,0);
\draw [ultra thick,->] (5,1) -- (5,0.1);

\node at (2.5,-2) {$\Gamma_{N \wedge M}$};
\end{tikzpicture}
\end{center}
\end{eg}

\begin{eg}
The two wheels $\Gamma_N$ and $\Gamma_M$ yield $\Gamma_{N \wedge M}$ pictured below, which consists of a single wheel as $\gcd(3,2)=1$. This wheel contains a cycle of $\lcm(3,2)=6$ vertices.

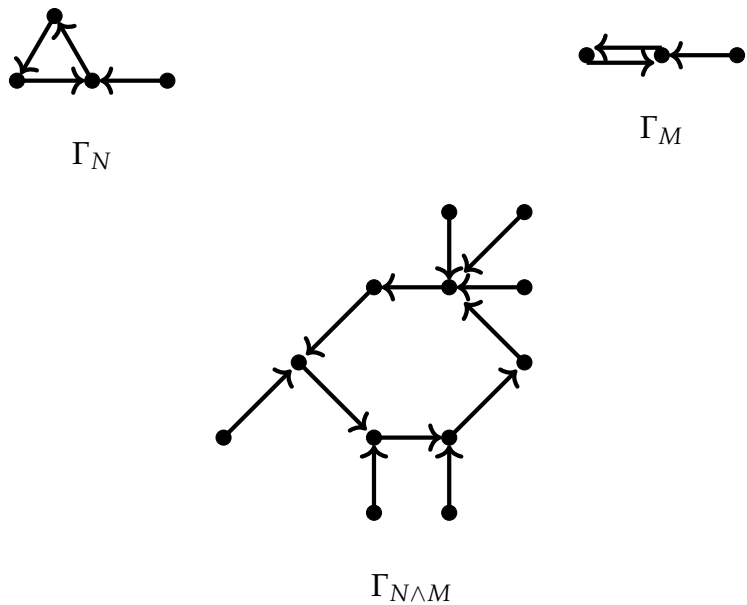
\begin{minipage}{.5\textwidth}
\begin{center}
\begin{tikzpicture}
\draw [fill] (-1,0) circle [radius=0.1];
\draw [fill] (0,0) circle [radius=0.1];
\draw [fill] (1,0) circle [radius=0.1];
\draw [fill] (-0.5,0.86) circle [radius=0.1];

\draw [ultra thick,->] (1,0) -- (0.1,0);
\draw [ultra thick,->] (-1,0) -- (-0.1,0);
\draw [ultra thick,->] (0,0) -- (-.45,.78);
\draw [ultra thick,->] (-0.5,0.86) -- (-0.95,0.086);

\node at (0,-1) {$\Gamma_N$};
\end{tikzpicture}
\end{center}
\end{minipage}
\begin{minipage}{.5\textwidth}
\begin{center}
\begin{tikzpicture}
\draw [fill] (-1,0) circle [radius=0.1];
\draw [fill] (0,0) circle [radius=0.1];
\draw [fill] (1,0) circle [radius=0.1];

\draw [ultra thick,->] (-1,-0.1) -- (-.1,-0.1);
\draw [ultra thick,->] (0,0.1) -- (-0.9,0.1);
\draw [ultra thick,->] (1,0) -- (0.1,0);

\node at (0,-1) {$\Gamma_M$};
\end{tikzpicture}
\end{center}
\end{minipage}

\begin{center}
\begin{tikzpicture}
\draw [fill] (0,0) circle [radius=0.1];
\draw [fill] (1,0) circle [radius=0.1];

\draw [fill] (-2,1) circle [radius=0.1];
\draw [fill] (0,1) circle [radius=0.1];
\draw [fill] (1,1) circle [radius=0.1];

\draw [fill] (-1,2) circle [radius=0.1];
\draw [fill] (2,2) circle [radius=0.1];

\draw [fill] (0,3) circle [radius=0.1];
\draw [fill] (1,3) circle [radius=0.1];
\draw [fill] (2,3) circle [radius=0.1];

\draw [fill] (1,4) circle [radius=0.1];
\draw [fill] (2,4) circle [radius=0.1];

\draw [ultra thick,->] (0,0) -- (0,0.9);
\draw [ultra thick,->] (1,0) -- (1,0.9);

\draw [ultra thick,->] (-2,1) -- (-1.1,1.9);
\draw [ultra thick,->] (0,1) -- (0.9,1);
\draw [ultra thick,->] (1,1) -- (1.9,1.9);

\draw [ultra thick,->] (-1,2) -- (-0.1,1.1);
\draw [ultra thick,->] (2,2) -- (1.2,2.8);

\draw [ultra thick,->] (0,3) -- (-0.9,2.1);
\draw [ultra thick,->] (1,3) -- (0.1,3);
\draw [ultra thick,->] (2,3) -- (1.1,3);

\draw [ultra thick,->] (1,4) -- (1,3.1);
\draw [ultra thick,->] (2,4) -- (1.2,3.2);

\node at (0.5,-1) {$\Gamma_{N \wedge M}$};

\end{tikzpicture}
\end{center}
\end{eg}

We end this section by collecting a couple of observations regarding the structure of $\Ksplit(\mtmod)$.
\begin{enumerate}
\item  $\Ksplit(\mtmod)$ is a $\ZZ_{\geq 0}$-graded commutative ring, with $deg([M]) = dim(M)$ for $[M] \in \on{Iso}(\mtmod)$. 
\item
\[
\mc{N} := \{ \sum_i a_i [M_i] \vert M_i \textrm{ is nilpotent } \} \subset \Ksplit(\mtmod) 
\]
is a graded ideal.The quotient 
\[
 \Ksplit(\mtmod)/ \mc{N}
\]
can be naturally identified with the integral span of wheels, with product given by $\wedge$. 
\end{enumerate}

\subsection{The homomorphism $\bch: \Ksplit(\Amod) \ra \Ksplit(\Mod_{k[t]})$} \label{bch}

In this subsection we study the ring homomorphism $\bch: \Ksplit(\Amod) \ra \Ksplit(\Mod_{k[t]})$ where $k$ is an field containing all roots of unity. For $[M] \in \on{Iso}(\mtmod)$, 
$\bch([M])$ is the isomorphism class of the $k[t]$-module $M \otimes_{\fun} k$ with basis $m \in M, m \neq 0$, and $t$-action extended $k$-linearly from $M$. In what follows, we will denote $M \otimes_{\fun} k$ by $M_k$ and the linear transformation $t \in \on{End}(M_k)$ by $T_M$. Fixing an ordering $m_1, \cdots, m_{dim(M)}$ of the non-zero elements of $M$ produces a basis for $M_k$, and the matrix of $T_M$ in this basis is the adjacency matrix $\Adj(\Gamma_M)$ of $\Gamma_M$. 

The isomorphism classes of indecomposable $k[t]$-modules correspond to $n \times n$ Jordan blocks $J_n (\lambda)$ with eigenvalue $\lambda$:
\begin{center} 
\begin{tikzpicture} 
\matrix (m) [matrix of math nodes,nodes in empty cells,right delimiter={]},left delimiter={[} ]{
\lambda&	1&	&	0\\
&	&	&	\\
&	&	&	1\\
0&	&	&	\lambda \\
};
\draw[loosely dotted] (m-1-1)-- (m-4-4);
\draw[loosely dotted] (m-1-2)-- (m-3-4);
\draw[loosely dotted] (m-1-1)-- (m-4-1);
\draw[loosely dotted] (m-1-2)-- (m-1-4);
\draw[loosely dotted] (m-1-1)-- (m-4-4);
\draw[loosely dotted] (m-1-4)-- (m-3-4);
\draw[loosely dotted] (m-4-1)-- (m-4-4);
\end{tikzpicture}
\end{center}

Describing $\bch$ thus amounts to decomposing $(M_k, T_M)$, or equivalently the adjacency matrix $\Adj(\Gamma_M)$, into Jordan blocks. It is clearly sufficient to consider the case where $\Gamma_M$ is connected.

 \begin{minipage}{.5\textwidth}
\begin{center}
\begin{tikzpicture}
\draw [ultra thick,->] (2,0) -- (2.0,0.9);
\draw [fill] (2,0) circle [radius=0.1];
\draw [fill] (2,2) circle [radius=0.1];
\draw [fill] (2,1) circle [radius=0.1];
\draw [ultra thick,->] (2,1) -- (2,1.9);
\draw [fill] (2,3) circle [radius=0.1];
\draw [ultra thick,->] (2,2) -- (2,2.9);
\draw [ultra thick,->] (2,3) -- (2,3.9);
\draw [fill] (2,4) circle [radius=0.1];
\node at (2,-1) {Ladder};
\end{tikzpicture}
\end{center}
\end{minipage}
\begin{minipage}{.5\textwidth}
\begin{center}
\begin{tikzpicture}
\draw [ultra thick,->] (0,0) -- (0.9,0);
\draw [ultra thick,->] (1,0) -- (1,0.9);
\draw [ultra thick,->] (1,1) -- (0.1,1);
\draw [ultra thick,->] (0,1) -- (0,0.1);
\draw [fill] (0,0) circle [radius=0.1];
\draw [fill] (1,1) circle [radius=0.1];
\draw [fill] (1,0) circle [radius=0.1];
\draw [fill] (0,1) circle [radius=0.1];
\node at (0.5,-2) {Simple cycle};
\end{tikzpicture}
\end{center}
\end{minipage}

The Jordan forms of $\Adj(\Gamma_M)$ when $M$ is a ladder tree of height $n-1$ or a simple cycle of length $n$ are easily seen to be the matrices $J_n(0)$ and $D_n$:

 \begin{minipage}{.5\textwidth}
\begin{center}
\begin{tikzpicture}
\matrix (m) [matrix of math nodes,nodes in empty cells,right delimiter={]},left delimiter={[} ]{
0&	1&	&	0\\
&	&	&	\\
&	&	&	1\\
0&	&	&	0\\
};
\draw[loosely dotted] (m-1-1)-- (m-4-4);
\draw[loosely dotted] (m-1-2)-- (m-3-4);
\draw[loosely dotted] (m-1-1)-- (m-4-1);
\draw[loosely dotted] (m-1-2)-- (m-1-4);
\draw[loosely dotted] (m-1-1)-- (m-4-4);
\draw[loosely dotted] (m-1-4)-- (m-3-4);
\draw[loosely dotted] (m-4-1)-- (m-4-4);
\node at (-2.2,0) {$J_n (0) =$};
\end{tikzpicture}
\end{center}
\end{minipage}
 \begin{minipage}{.5\textwidth}
 \begin{center}
\begin{tikzpicture})]
\matrix (m) [matrix of math nodes,nodes in empty cells,right delimiter={]},left delimiter={[} ]{
\zeta&	&	0\\
&	&	\\
0&	&	\zeta^{n}\\
};
\draw[loosely dotted] (m-1-1)-- (m-3-3);
\node at (-2.2,0) {$D_n = $};
\end{tikzpicture}
\end{center}
\end{minipage}
with $\zeta=e^{\frac{2\pi i}{n}}$

For more general directed graphs arising as $\Gamma_M$, this problem is solved in \cite{CT}. We proceed to recall the solution given there, specialized to our setup. 

\begin{dfn}
 A \emph{partition} of $\Gamma_M$ is a collection $\{C_1, \cdots, C_r, Z_1, \cdots, Z_s \}$ of disjoint chains $C_1, \cdots, C_r$ and cycles $Z_1, \cdots, Z_s$ whose union is $M \backslash 0$. A \emph{proper partition} of $M$ is a partition satisfying the following two additional properties:
 \begin{enumerate}
 \item Each cycle in $M$ is equal to one of $Z_1, \cdots, Z_s$.
 \item For each $1 \leq i \leq r$, if $\Gamma^{i}_M$ is the graph obtained from $\Gamma_M$ by deleting all of the vertices in $Z_1, \cdots Z_s, C_1, \cdots C_i$, then $C_{i+1}$ is a chain of maximal length in $\Gamma^{i}_M$. 
 \end{enumerate}
\end{dfn}

It is easy to see that proper partitions of $\Gamma_M$ exist, and can be obtained as follows. Each connected component of $\Gamma_M$ has at most one (and necessarily unique) cycle - take these to be $Z_1, \cdots, Z_s$, Upon deleting the $Z_j, \; 1 \leq j \leq s$, we are left with a forest of rooted trees. We now look for the longest chain $C_1$ in this forest, delete it, and repeat, obtaining $C_2, \cdots C_r$. 

\begin{eg} In the graph $\Gamma_M$ below, 
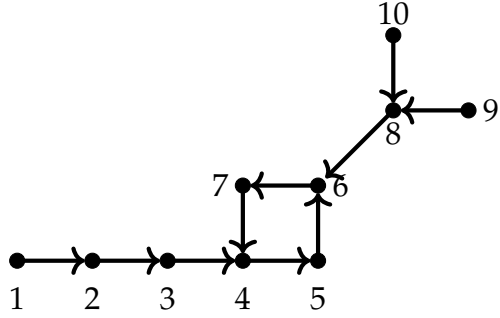
\begin{center}
\begin{tikzpicture}
\draw [ultra thick,->] (0,0) -- (0.9,0);
\draw [ultra thick,->] (1,0) -- (1,0.9);
\draw [ultra thick,->] (1,1) -- (0.1,1);
\draw [ultra thick,->] (0,1) -- (0,0.1);
\draw [ultra thick,->] (-3,0) -- (-2.1,0);
\draw [ultra thick,->] (-2,0) -- (-1.1,0);
\draw [ultra thick,->] (-1,0) -- (-0.1,0);
\draw [ultra thick,->] (2,2) -- (1.1,1.1);
\draw [ultra thick,->] (2,3) -- (2,2.1);
\draw [ultra thick,->] (3,2) -- (2.1,2);
\draw [fill] (0,0) circle [radius=0.1];
\draw [fill] (1,1) circle [radius=0.1];
\draw [fill] (1,0) circle [radius=0.1];
\draw [fill] (0,1) circle [radius=0.1];
\draw [fill] (-1,0) circle [radius=0.1];
\draw [fill] (-2,0) circle [radius=0.1];
\draw [fill] (-3,0) circle [radius=0.1];
\draw [fill] (2,2) circle [radius=0.1];
\draw [fill] (2,3) circle [radius=0.1];
\draw [fill] (3,2) circle [radius=0.1];
\node at (-3,-0.5) {1};
\node at (-2,-0.5) {2};
\node at (-1,-0.5) {3};
\node at (0,-0.5) {4};
\node at (1,-0.5) {5};
\node at (1.3,1) {6};
\node at (-0.3,1) {7};
\node at (2,1.7) {8};
\node at (3.3,2) {9};
\node at (2,3.3) {10};
\end{tikzpicture}
\end{center}
a proper partition is given by $\{C_1, C_2, C_3, Z_1 \}$, where $C_1 = \{ 1, 2, 3 \}$, $C_2 = \{ 9, 8 \}$, $C_3=\{ 10 \}$, and $Z_1 = \{ 4, 5, 6, 7 \}$.
\end{eg}

The following theorem describes the Jordan form of $\Adj(\Gamma_M)$.

\begin{theorem}[\cite{CT}]
Let $\{C_1, \cdots, C_r, Z_1, \cdots, Z_s \}$ be a proper partition of $\Gamma_M$ into chains $C_i$ of length $l(C_i)$ and cycles $Z_j$ of length $l(Z_j)$. Then
\[
\Adj(\Gamma_M) \simeq \bigoplus^r_{i=1} J_{l(C_i)}(0) \oplus \bigoplus^s_{j=1} D_{n} 
\] 
\end{theorem}

We are now able to characterize the image of the homorphism $\bch$:

\begin{theorem}
The image of $\bch$ is the subring of $\Ksplit(\mtmod)$ generated by $[J_n(0)], [D_n]$, $n \geq 1$. 
\end{theorem}

We note one final consequence of the fact that $\bch$ is monoidal. By the above discussion, $\bch(M)$ may be identified with the adjacency matrix of $\Gamma_M$. It follows that 
\[
\bch(M \wedge N) = \bch(M) \otimes_k \bch(N) 
\]
In other words, $\Adj(\Gamma_{M \wedge N}) = \Adj(\Gamma_M) \otimes \Adj (\Gamma_N)$, where $\otimes$ on the right denotes the Kronecker product of matrices. This is the defining property of the \emph{tensor product graph} $\Gamma_M \otimes \Gamma_N$ (see \cite{Weich}).  To summarize,

\begin{prop}
For $M, N \in \mtmod$, $\Gamma_{M \wedge N} = \Gamma_M \otimes \Gamma_N$. 
\end{prop} 

\section{Skew shapes and the monoids  $\mxn$}

In this section we consider a subcategory $\Sk_n \subset \mxnmod$ (originally introduced in \cite{SzSk}) consisting of $n$-dimensional skew shapes. Our goal is to give an explicit description of the product in the ring $\Ksplit(\Sk_n)$. 

\subsection{Skew shapes and $\Pn$-modules} \label{skew}

$\mathbb{Z}^n$ has a natural partial order where for $$ x=(x_1, \cdots, x_n) \in \mathbb{Z}^n  \textrm{ and } y = (y_1, \cdots, y_n) \in \mathbb{Z}^n, $$ $$ x \leq y \iff x_i \leq y_i \textrm{ for } i=1, \cdots, n. $$

\begin{dfn}
An \emph{n-dimensional skew shape} is a finite convex sub-poset  $S \subset \mathbb{Z}^n$. $S$ is \emph{connected} iff the corresponding poset is. 
We consider two skew shapes $S, S'$ to be equivalent iff they are isomorphic as posets. If $S, S'$ are connected, then they are equivalent iff $S'$ is a translation of $S$, i.e. if there exists $a \in \mathbb{Z}^n$ such that $S'=a+S$. 
\end{dfn}

The condition that $S$ is connected is easily seen to be equivalent to the condition that any two elements of $S$ can be connected via a lattice path lying in $S$. The name \emph{skew shape} is motivated by the fact that for $n=2$, a connected skew shape in the above sense corresponds (non-uniquely) to a difference $\lambda \backslash \mu$ of two Young diagrams in French notation. 

\begin{eg} \label{example_a}
Let $n=2$, and $$S \subset \mathbb{Z}^2 = \{ (1,0),(2,0),(3,0),(0,1),(1,1), (0,2) \}$$ (up to translation by $a \in \mathbb{Z}^2).$ Then $S$ corresponds to the connected skew Young diagram  

%\vspace{0.2in}
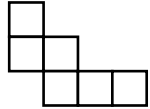
\begin{center}
\Ylinethick{1pt}
\gyoung(;,;;,:;;;)
\end{center}
\end{eg}

Let $S \subset \mathbb{Z}^n$ be a skew shape. We may attach to $S$ a $\Pn$-module $M_S$ with underlying set
\[
M_S = S \sqcup \{ 0 \},
\]
and action of $\Pn$ defined by
\begin{align*}
x^e \cdot s &= \begin{cases} s+ e, & \mbox{if } s+e \in S \\ 0 & \mbox{otherwise} \end{cases} \\
& e \in \mathbb{Z}^n_{\geq 0}, s \in {S}.
\end{align*}
 %It is clear that $M_S$ is a $\Pn$-module of type $\alpha$, and that $\m^k \cdot M_S =0$ for $k$ sufficiently large, where $\m = (x_1, \cdots, x_n)$ is the maximal ideal.  
In particular, $x_i \cdot s = s+e_i$ if $s+e_i \in S$, $0$ otherwise, where $e_i$ is the $i$th standard basis vector. 
$M_S$ is a graded $\PP_n$-module with respect to its $\mathbb{Z}^n_{\geq 0}$-grading, in which $deg(x_i)=e_i$ - the $i$th standard basis vector. 

\begin{eg} Let $S$ as in Example \ref{example_a}. $x_1$ (resp. $x_2$) act on the $\mathbb{P}_2 =  \langle x_1, x_2 \rangle$-module $M_S$ by moving one box to the right (resp. one box up) until reaching the edge of the diagram, and $0$ beyond that. A minimal set of generators for $M_S$ is indicated by the black dots:

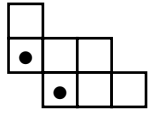
\begin{center}
\Ylinethick{1pt}
\gyoung(;,\bullet;;,:;\bullet;;)
\end{center}
\end{eg}

We may consider the subcategory $\Sk_n \subset \Pn - \on{mod}$ consisting of $\Pn$-modules $M$ satisfying the following two conditions:
\begin{enumerate}
\item $M$ admits a $\mathbb{Z}^n$--grading.
\item For $a \in \Pn$, $m_1, m_2 \in M$,
\[
a \cdot m_1 = a \cdot m_2 \iff m_1 = m_2 \textrm{ OR } a \cdot m_1 = a \cdot m_2 =0 
\]
\end{enumerate}

The following proposition follows from results in \cite{SzSk}:

\begin{prop}
$\Sk_n$ forms a full monoidal subcategory of $\mxnmod$. If $M \in \Sk_n$ is indecomposable, then $M \simeq M_S$ for a connected skew shape $S$. 
\end{prop}

In other words, given connected skew shapes $S_1, S_2$, the $\Pn$-module $M_{S_1} \wedge M_{S_2}$ is isomorphic to $\oplus M_{U_j}$, where $U_j$ are connected skew shapes. 

%Hence, if $S$ an $n$ dimensional skew shape $S$ and $T$ is a set with the action $\myn$, we say $S$ is isomorphic to $T$ (or $S \simeq T$) if there exists a bijection $\phi:M \rightarrow S$ such that $\phi(x_i(p)) = y_i(\phi(p))$ whenever $x_i(p) \neq 0$.

\begin{lemma} \label{skew_int}
If $S_1, S_2 \in \Sk_n$ with chosen embeddings in $\mathbb{Z}^n$, and $t \in \mathbb{Z}^n$, then 
\begin{equation*}
S_1 \cap (S_2 + t)
\end{equation*}
is also an $n$ dimensional skew shape, possibly empty.
\end{lemma}

\begin{proof} As $S_2$ is a skew shape, so is $S_2 + t$. Hence, it suffices to show the intersection of skew shapes is a skew shape, that is, $S_1 \cap S_2$ is a skew shape.

It is immediate that $S_1 \cap S_2$ is a finite poset of $\mathbb{Z}^n$. Further, if $a,b,c \in S_1 \cap S_2$ and $a \leq c \leq b$, then as both $S_1$ and $S_2$ are convex, $c \in S_1 \cap S_2$. Hence, $S_1 \cap S_2$ is convex and therefore a skew shape.
\end{proof}

\begin{theorem}
If $S_1, S_2 \in \Sk_n$ with chosen embeddings in $\ZZ^n$  then
\begin{equation*}
M_{S_1} \wedge M_{S_2} = \bigoplus_{t \in \mathbb{Z}^n} M_{S_1 \cap (S_2 + t)}
\end{equation*}
\end{theorem}

\begin{rmk}
Since $S_1, S_2$ are finite embedded skew shapes, the intersection $S_1 \cap (S_2 + t)$ is empty for all but finitely many $t \in \ZZ^n$. Moreover, by Lemma \ref{skew_int}, the right hand side is an object in $\Sk_n$. 
\end{rmk}

\begin{proof} We will use the notation $a_t \in M_{S_1 \cap (S_2+t)}$ to denote an element occurring in the $t$-th summand in $\bigoplus_{t \in \mathbb{Z}^n} M_{S_1 \cap (S_2 + t)}$.
Define $$\Psi:  M_{S_1} \wedge M_{S_2} \ra \bigoplus_{t \in \mathbb{Z}^n} M_{S_1 \cap (S_2 + t)} $$
by 
$$  \Psi((a,b)) = a_{a-b} \in M_{S_1 \cap (S_2 + a-b)} $$
We proceed to show that $\Psi$ is an isomorphism of $\Pn$-modules. $\Psi$ is clearly injective, and sends $0$ to $0$. Moreover, if $a_{t} \in M_{S_1 \cap (S_2 + t)}$ is nonzero, then $a=b+t$ for some nonzero $b \in S_2$, hence $a_t=\Psi((a,b))$. $\Psi$ is therefore a bijection. 

It remains to check that $\Psi$ is morphism of $\Pn$-modules, or equivalently that $\Psi \circ x_i = x_i \circ \Psi$ for $i=1, \cdots, n$. 

%We note elements in $\bigoplus_{t \in \mathbb{Z}^n} M_{S_1 \cap (S_2 + t)}$ have form $(a,t)$, for $a \in S_1 \cap (S_2 + t)$, or are equal to the basepoint $0$. To show the above representations are identical, we show the map $\phi$ on $(S_1 \times S_2)\cup\{0\}$ defined by $\phi((a,b)) = (a, a-b)$ and $\phi(0)=0$ is an isomorphism, in the sense that it is a bijection satisfying $\phi \circ x_i(0)=x_i \circ \phi(0)$ and
%\begin{equation}
%\phi \circ x_i((a,b))=
%\begin{cases}
%x_i \circ \phi((a,b))		&x_i((a,b))=0\\
%(x_i \circ \phi_1(a,b),\phi_2(a,b))	&\text{otherwise}

%\end{cases}
%\end{equation}
%where $\phi_1$ and $\phi_2$ return the first and second coordinates of $\phi$.

%$\phi$ is clearly injective. It is also onto as $\phi(0)=0$ and for each $(a,t) \in \bigoplus_{t \in \mathbb{Z}^n} M_{S_1 \cap (S_2 + t)}$ we observe that $a \in S_1 \cap (S_2 + t)$. Thus $a=b+t$ for some $b \in S_2$, and so $\phi((a,b))=(a,t)$. Hence $\phi$ is bijective.

Suppose $(a,b)$ is a non-zero element in the domain of $\Psi$. If $x_i((a,b))=0$, then either $x_i(a)=0$ or $x_i(b)=0$, or equivalently, either $a+e_i \notin S_1$ or $b+e_i \notin S_2$. Thus $a+e_i \notin S_1 \cap (S_2 + a-b)$ and so $x_i \cdot a_{a-b} = x_i \circ \Psi((a,b))=0=\Psi \circ x_i((a,b))$.
 
Otherwise, $x_i((a,b))=(a+e_i,b+e_i) \in S_1 \times S_2$ and so it follows that $\Psi \circ x_i((a,b))=(a+e_i)_{a-b}$. Meanwhile, $\Psi(a,b)=a_{a-b}$. As $a+e_i \in S_1$, $b+e_i \in S_2$, we have $a+e_i \in S_1\cap (S_2 + a-b)$, and so $x_i \cdot a_{a-b}=(a+e_i)_{a-b}$. Hence
\begin{equation*}
x_i \circ \Psi((a,b)) = \Psi \circ x_i \cdot (a,b)
\end{equation*}
This completes the proof.
\end{proof}

\medskip

\begin{rmk}
The situation can be visualized as follows. For two embedded skew shapes $S$ and $T$, the connected component of the skew shape in $M_{S} \wedge M_{T}$ containing some point $(a,b)$ is the intersection of $S$ with the the unique translate of $T$ that makes $a$ and $b$ coincide . Below is an example of $S$, $T$ and their intersection in red for $n=2$.

\begin{minipage}{.33\textwidth}
\begin{center}
\Ylinethick{1pt}
\gyoung(;,;;\bullet;,:;!<\Ylinethick{1pt}>;;)
\end{center}
\end{minipage}
\begin{minipage}{.33\textwidth}
\begin{center}
\Ylinethick{1pt}
\gyoung(;;;\bullet;;,:::;;)
\end{center}
\end{minipage}
\begin{minipage}{.33\textwidth}
\begin{center}
\Ylinethick{1pt}
\gyoung(:;,;!<\Yfillcolor{red}>;;\bullet;!<\Yfillcolor{white}>;,::;!<\Yfillcolor{red}>;;)
\end{center}
\end{minipage}
\end{rmk}

\medskip

\begin{eg}
Suppose the we have the following skew shapes $S$ and $T$ in $n=2$ dimensions.

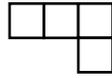
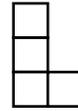
\begin{minipage}{.5\textwidth}
\begin{center}
\Ylinethick{1pt}
\gyoung(;;;,::;)
\end{center}
\end{minipage}
\begin{minipage}{.5\textwidth}
\begin{center}
\Ylinethick{1pt}
\gyoung(;,;,;;)
\end{center}
\end{minipage}

To find the collection of skew shapes occurring in $M_S \wedge M_T$ we observe the nontrivial intersections of $S$ and $T$ under translation are given below with regions of intersection in red, and regions of nonintersection in yellow.

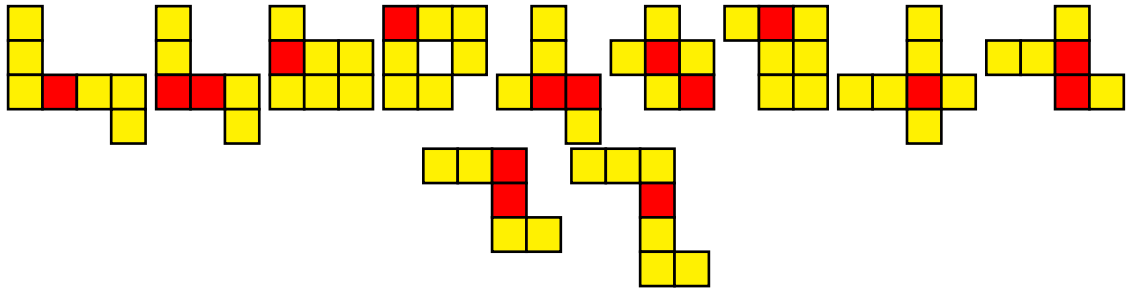
\begin{center}
\Ylinethick{1pt}
\Yfillcolor{yellow}
\gyoung(;,;,;!<\Yfillcolor{red}>;!<\Yfillcolor{yellow}>;;,:::;)
\gyoung(;,;,!<\Yfillcolor{red}>;;!<\Yfillcolor{yellow}>;,::;)
\gyoung(;,!<\Yfillcolor{red}>;!<\Yfillcolor{yellow}>;;,;;;)
\gyoung(!<\Yfillcolor{red}>;!<\Yfillcolor{yellow}>;;,;:;,;;)
\gyoung(:;,:;,;!<\Yfillcolor{red}>;;!<\Yfillcolor{yellow}>,::;)
\gyoung(:;,;!<\Yfillcolor{red}>;!<\Yfillcolor{yellow}>;,:;!<\Yfillcolor{red}>;)
\gyoung(;!<\Yfillcolor{red}>;!<\Yfillcolor{yellow}>;,:;;,:;;)
\gyoung(::;,::;,;;!<\Yfillcolor{red}>;!<\Yfillcolor{yellow}>;,::;)
\gyoung(::;,;;!<\Yfillcolor{red}>;,::;!<\Yfillcolor{yellow}>;)
\gyoung(;;!<\Yfillcolor{red}>;,::;!<\Yfillcolor{yellow}>,::;;)
\gyoung(;;;,::!<\Yfillcolor{red}>;!<\Yfillcolor{yellow}>,::;,::;;)
\end{center}

It follows that $M_S \wedge M_T$ decomposes into indecomposable modules corresponding to the following skew shapes with the indicated multiplicities:

\begin{center}
\Ylinethick{1pt}
\Yfillcolor{red}
8 \gyoung(;)
$\oplus$ 2 \gyoung(;;)
$\oplus$ 2 \gyoung(;,;)
\end{center}
\end{eg}

Note that we further decomposed the disconnected skew shape 
\begin{center}
\Ylinethick{1pt}
\Yfillcolor{red}
\gyoung(;,:;)
\end{center}
into its connected components.

\address{\tiny DEPARTMENT OF MATHEMATICS AND STATISTICS, BOSTON UNIVERSITY, 111 CUMMINGTON MALL, BOSTON} \\
\indent \footnotesize{\email{dbeers@bu.edu}}

\address{\tiny DEPARTMENT OF MATHEMATICS AND STATISTICS, BOSTON UNIVERSITY, 111 CUMMINGTON MALL, BOSTON} \\
\indent \footnotesize{\email{szczesny@math.bu.edu}}

\end{document}